\documentclass[12pt]{amsart}
\usepackage{amscd}
\def\frk{\frak}               

\def\Phi{{\frk n}}
\def\Phi{{\frk N}}
\def\opn#1#2{\def#1{\operatorname{#2}}} 
\opn\chara{char} \opn\length{\ell} \opn\pd{pd} \opn\rk{rk}
\opn\projdim{proj\,dim} \opn\injdim{inj\,dim} \opn\rank{rank}
\opn\depth{depth} \opn\grade{grade} \opn\height{height}
\opn\embdim{emb\,dim} \opn\codim{codim}

\opn\Tr{Tr} \opn\bigrank{big\,rank}
\opn\superheight{superheight}\opn\lcm{lcm}
\opn\trdeg{tr\,deg}
\opn\reg{reg} \opn\lreg{lreg} \opn\ini{in} \opn\lpd{lpd}
\opn\size{size}
\opn\div{div} \opn\Div{Div} \opn\cl{cl} \opn\Cl{Cl}
\opn\Spec{Spec} \opn\Supp{Supp} \opn\supp{supp} \opn\Sing{Sing}
\opn\Ass{Ass} \opn\Min{Min}
\opn\Ann{Ann} \opn\Rad{Rad} \opn\Soc{Soc}
\opn\Im{Im} \opn\Ker{Ker} \opn\Coker{Coker} \opn\Am{Am}
\opn\Hom{Hom} \opn\Tor{Tor} \opn\Ext{Ext} \opn\End{End}
\opn\Aut{Aut} \opn\id{id}

\opn\nat{nat}
\opn\pff{pf}
\opn\Pf{Pf} \opn\GL{GL} \opn\SL{SL} \opn\mod{mod} \opn\ord{ord}
\opn\Gin{Gin} \opn\Hilb{Hilb}
\opn\aff{aff} \opn\con{conv} \opn\relint{relint} \opn\st{st}
\opn\lk{lk} \opn\cn{cn} \opn\core{core} \opn\vol{vol}
\opn\link{link} \opn\star{star}
\opn\gr{gr}

\def\pot#1#2{#1[\kern-0.28ex[#2]\kern-0.28ex]}

\opn\dirlim{\underrightarrow{\lim}}
\opn\inivlim{\underleftarrow{\lim}}
\def\Implies{\ifmmode\Longrightarrow \else
        \unskip${}\Longrightarrow{}$\ignorespaces\fi}
\def\implies{\ifmmode\Rightarrow \else
        \unskip${}\Rightarrow{}$\ignorespaces\fi}
\def\iff{\ifmmode\Longleftrightarrow \else
        \unskip${}\Longleftrightarrow{}$\ignorespaces\fi}

\let\:=\colon
\newtheorem{Theorem}{Theorem}[section]
\newtheorem{Lemma}[Theorem]{Lemma}

\newtheorem{Remark}[Theorem]{Remark}

\newtheorem{Example}[Theorem]{Example}

\let\epsilon\varepsilon
\let\phi=\varphi
\let\kappa=\varkappa
\def\qed{\ifhmode\textqed\fi
      \ifmmode\ifinner\quad\qedsymbol\else\dispqed\fi\fi}
\def\textqed{\unskip\nobreak\penalty50
       \hskip2em\hbox{}\nobreak\hfil\qedsymbol
       \parfillskip=0pt \finalhyphendemerits=0}
\def\dispqed{\rlap{\qquad\qedsymbol}}

\opn\dis{dis}
\def\pnt{{\raise0.5mm\hbox{\large\bf.}}}

\opn\Lex{Lex}

\begin{document}

\title{Janet's Algorithm}

\author{Imran Anwar}

\thanks{The author is highly grateful to the Abdus Salam School of Mathematical Sciences,
  GC University, Lahore, Pakistan in supporting and facilitating this
  research. The author  would like to thank   Prof. J\"{u}rgen Herzog for introducing the
  idea and encouragement.}

\address{Imran Anwar, Abdus Salam School of Mathematical Sciences, 68-B New Muslim Town,
    Lahore,Pakistan.}\email{iimrananwar@gmail.com}
 \maketitle

\begin{abstract}
We have introduced the Janet's algorithm for the Stanley
decomposition of a monomial ideal $I\subset S = K[x_1, . . . ,x_n]$
and prove that Janet's algorithm gives the squarefree Stanley
decomposition of $S/I$ for a squarefree monomial ideal $I$. We have
also shown that the Janet's algorithm gives a partition of a
simplicial complex.
 \vskip 0.4 true cm
 \noindent
  {\it Key words } : Stanley decomposition, squarefree Stanley decomposition,  partition of a simplicial complex.\\
 {\it 2000 Mathematics Subject Classification}: Primary 13P10, Secondary
13H10, 13F20, 13C14.\\
\end{abstract}

\section{Introduction}

Let $K$ be a field, $S = K[x_1, . . . ,x_n]$ the polynomial ring in
$n$ variables. Let $u \in S$ be a monomial and $Z$ a subset of
$\{x_1, . . . ,x_n\}$. We denote by $uK[Z]$ the $K$-subspace of $S$
whose basis consists of all monomials $uv$ where $v$ is a monomial
in $K[Z]$. The $K$-subspace $uK[Z] \subset S$ is called a Stanley
space of dimension $|Z|$. Stanley decomposition has been discussed
in various combinatorial and algebraic contexts
 see \cite{AP}, \cite{AP1}, \cite{Ap}, \cite{HJY}, \cite{HT}, \cite{MS} and \cite{Ra}.  \\

Let $I \subset S$ be a monomial ideal, and denote by $I^c \subset S$
the $K$-linear subspace of $S$ spanned by all monomials which do not
belong to $I$. Then $S = I^c \oplus I$ as a $K$-vector space, and
the residues of the monomials in $I^c$ form a $K$-basis of $S/I$.
 One way to obtain the Stanley decomposition for $S/I$ is prime filtration
for instance see proof of \cite[Theorem 6.5]{HP}, but not all the
Stanley decompositions can be obtained from prime filtrations see
\cite{MS} and \cite{Ja}.\\

Let $\Delta$ be a simplicial complex of dimension $d-1$ on the
 vertex set $V = {x_1, . . . ,x_n}$. A subset $\mathcal{I} \in \Delta$
 is called an interval, if there exists faces
$F,G \subset \Delta$ such that $\mathcal{I} = \{H \in \Delta : F
\subseteq H \subseteq G\}$. We denote this interval given by $F$ and
$G$ also by $[F,G]$ and call $dim\ (G) - dim\ (F)$ the rank of the
interval. A partition $\mathcal{P}$ of $\Delta$ is a presentation of
$\Delta$ as a disjoint union of intervals. The $r$-vector of
$\mathcal{P}$ is the integer vector $r = ( r_0, r_1, . . . , r_d ) $
where $r_i$ is the number of intervals of rank $i$. Let $\Delta$ be
a simplicial complex and $\mathcal{F}(\Delta)$ its set of facets.
Stanley calls a simplicial complex $\Delta$ partitionable if there
exists a partition $\Delta = \bigcup_{i=1}^r [F_i,G_i]$ with
$\mathcal{F}(\Delta) = \{G_1, . . . ,G_r\}$. We call a partition
with this property a nice partition. If a Cohen Macaulay simplicial
complex $\Delta$ is  partitionable then the square free ideal
$I_{\Delta}$ will be a Stanley ideal see \cite[corollary
3.5]{HJY}.\\

 We have described the Janet's algorithm to obtain the
Stanley decomposition of a monomial ideal $I$. More importantly to
obtain a square free Stanley decomposition of  $I^c$
 for  a square free ideal $I$ see Lemma \ref{12}. When we have an algorithm for the
  squarefree Stanley decomposition of $I^c$ then from \cite[Proposition
 3.2]{HJY} we get a motivation to develop the Janet's algorithm
 for the partition of a simplicial complex $\Delta$ see Lemma \ref{23}.\\

Here I would like to give a short description on the history of this
subject. The French mathematician Maurice Janet presented an
algorithm to construct  a  special basis (Janet's Basis) for a
finitely generated module over $K<\partial_1,...,\partial_n >$
(where $K$ is a differential field and $\partial_i$'s are partial
derivatives) after a longer visit to Hilbert in G\"{o}ttingen in the
early twenties of the last century, cf.\cite{Jn}, \cite{Jn1}.
 Independently W. Gr\"{o}bner introduced a device now a days known
 as Gr\"{o}bner basis, to
compute in residue class rings of polynomial rings in the late
thirties, cf.\cite{Gr}, \cite{Gr1}, at that time restricted to the
zero-dimensional case. In the 1960s, Gr\"{o}bner basis techniques to
compute with modules over the polynomial ring had an enormous boom
as a consequence of both, B.Buchberger's thesis constructing
Gr\"{o}bner bases, and the general development of powerful computing
devices. By 1980, F.-O. Schreyer proved that
Buchberger's so called S-polynomial come very close to a Gr\"{o}bner
bases  of the syzygy module. After Janet work has been ignored by
the mathematical community more than fifty years, J.-F. Pommaret,
working on Spencer cohomology, became aware of Janet's work and
pointed out that Janet's algorithm when applied to linear partial
differential equations with constant coefficients is a variant of
Buchberger's algorithm and the Janet bases is a special case of
Gr\"{o}bner bases in this case, though Janet's philosophy is
completely different from Gr\"{o}bner's philosophy. V. Gerdt and
collaborators have shown that Janet's constructive ideas lead to
very effective methods. They created an axiomatic framework for
Janet's approach called involutive division algorithm. For instance,
the Singular package, recently has started to use the Janet or
involutive division algorithm to construct the Gr\"{o}bner bases.\\

\section{Janet's algorithm and Stanley decomposition}

 In this section, I have given a description on the Janet's algorithm for
the Stanley decompositions, note that it is a recursive procedure to
find the Stanley decomposition. Also Janet's algorithm give a unique
Stanley decomposition after fixing the order of the variables.\\

\begin{Lemma}
Let $I\subset S=K[x_1,x_2,...,x_n]$ be a monomial ideal, Janet's
algorithm gives a Stanley decomposition of I.
\end{Lemma}
\begin{proof}
By Janet's algorithm, we can write
\begin{center}
$I\cap x_n^{k}K[x_1, x_2,...,x_{n-1}]=x_n^{k}I_{k}$
\end{center}
where $I_k\subset K[x_1, x_2,...,x_{n-1}]$ is a monomial ideal and
from construction it is clear that $$I_0\subseteq I_1 \subseteq \ \
. \ \ . \ \ . \ \ I_k\subseteq I_{k+1}\subseteq .\ .$$ Let us define
$$\alpha \ =\  min\{k\ \ | I_k\neq 0 \}$$ and $$\beta \ =\  min\{k\ \ | I_k= I_{\gamma} \ for \ all \ \gamma \geq k  \}$$
there exists such a $\beta$ because $S'=K[x_1,x_2,...,x_{n-1}]$ is
Noetherian so the ascending chain of ideals mentioned above will
stabilize at some point.\\
We will prove it by using induction on $n$. \\ For $n=1$, it is
clear.\\ Suppose all the monomial ideals in
$S'=K[x_1,x_2,...,x_{n-1}]$ has a Stanley decomposition. Now
consider $I\subset S  =K[x_1,x_2,...,x_n]$, from above it is clear
that $$I=\bigoplus_{k} x_n^k I_k $$ where $I_k$ is a monomial ideal
in $S'= K[x_1, x_2,...,x_{n-1}]$ so it has a Stanley decomposition
as $$I_k=\bigoplus_{i_k=1}^{r_k} u_{i_k}K[Z_{i_k}] \ \ \ \ \ \ \ \ \
\ for \ all\ k.$$ Now by Janet's algorithm we have the Stanley
decomposition of $I$ as follows:
$$I=(\bigoplus_{\alpha \leq k < \beta} x_n^k I_k) \ \ \ \bigoplus \ \ \ (\bigoplus_{k\geq \beta}x_n^k I_k )$$
$$I=\bigoplus_{\alpha \leq k < \beta}( \bigoplus_{i_k=1}^{r_k} u_{i_k}x_n^k K[Z_{i_k}]) \ \ \ \bigoplus \ \ \ (\bigoplus_{i_{\beta}=1}^{r_{\beta}} u_{i_{\beta}}x_n^{\beta}K[Z_{i_{\beta}},x_n]),$$
which is a Stanley decomposition of $I$.
\end{proof}
A Stanley space $uK[Z]$ is called a squarefree Stanley space, if $u$
is a squarefree monomial and $supp(u) \subset Z$. Now we will show
that in the case of a square free monomial ideal $I$, Janet's
algorithm gives a square free Stanley decomposition recursively in
the following lemma;

\begin{Lemma}
If $I\subset S=K[x_1,x_2,...,x_n]$ is a square free monomial ideal,
Janet's algorithm gives a square free Stanley decomposition of $I$.
\end{Lemma}
\begin{proof}
From above lemma, we can write
\begin{center}
$I\cap x_n^{k}K[x_1, x_2,...,x_{n-1}]=x_n^{k}I_{k}$
\end{center}
define $\alpha$ and $\beta$ as above. For any square free monomial
ideal $I$ it is easy to see that $\alpha ,\beta \leq1$ since
$I_1=I_{\gamma}$ for all $\gamma\geq 1$. As we know that
$I_1\subseteq I_{\gamma}$ for $\gamma\geq1$, let us take a monomial
$u\in I_{\gamma}$ then  $u\in K[x_1,x_2, . . . ,x_{n-1}]$
so $ux_n^{\gamma }\in I$ $\Rightarrow $ $\sqrt{u}.x_n\in I$ $\Rightarrow $ $ux_n\in I$, hence $u\in I_1$.\\
We will prove it by using induction on $n$. \\For $n=1$ it is
trivial. Suppose every square free monomial ideal $I$ in
$S'=K[x_1,x_2, .\ . \ . \ ,x_{n-1}]$ has a square free Stanley decomposition.\\
Now take $I\subseteq S=K[x_1,x_2, .\ . \ . \ ,x_{n}]$, by Janet's
algorithm we can write$$I=\bigoplus_{k} x_n^k I_k \ ,$$ where each
$I_k$ is a square free monomial ideal in $S'$ and so,  it has a square free
Stanley decomposition as follows
$$I_k=\bigoplus_{i_k=1}^{r_k} u_{i_k}K[Z_{i_k}]\ , $$
for all $k$. Now by Janet's algorithm we have the Stanley
decomposition of $I$ as follows:
$$I=\bigoplus_{k\geq \alpha}x_n^k I_k $$
 Janet algorithm gives the Stanley decomposition for different cases
 as follows:\\
When $\alpha\neq \beta$ then $\alpha=0$ and $\beta=1$, so the
Stanley decomposition of $I$ will be of the form:
$$I=(\bigoplus_{i_{\alpha}=1}^{r_{\alpha}} u_{i_{\alpha}}K[Z_{i_{\alpha}}])\bigoplus (\bigoplus_{i_{\beta}=1}^{r_{\beta}} u_{i_{\beta}}x_n K[Z_{i_{\beta}},x_n])$$
as $supp (u_{i_{\beta}})\in \ Z_{i_{\beta}}$ $\Rightarrow$ $supp
(u_{i_{\beta}}x_n)\in \{Z_{i_{\beta}} ,x_{n} \}$ and
$u_{i_{\beta}}x_n$ remain square free as $u_{i_{\beta}}$ is square
free in $S'$. Hence it is a square free Stanley decomposition of
$I$.\\
When $\alpha = \beta (\leq 1)$, then Stanley decomposition of $I$
will be
$$I= \bigoplus_{i_{\beta}=1}^{r_{\beta}} u_{i_{\beta}}x_n^{\beta} K[Z_{i_{\beta}},x_n], $$
where  $\beta \leq 1$, this is clearly a square free Stanley decomposition.
\end{proof}

Now we will describe the Janet's algorithm for a squarefree Stanley
decomposition of $I^c$ when $I\subset S=K[x_1,x_2,...,x_n]$ is  a
squarefree monomial ideal.
\begin{Lemma}\label{12}
If $I\subset S=K[x_1,x_2,...,x_n]$ is a square free monomial ideal,
Janet's algorithm gives a square free Stanley decomposition of $I^c$
recursively.
\end{Lemma}
\begin{proof}
For a monomial ideal $I\subset S$, we can write
$$I_k^c x_n^k = I^c \cap x_n^k K[x_1,x_2, .\ . \ . \ ,x_{n-1}]=\ \ x_n^k (K[x_1,x_2, .\ . \ . \ ,x_{n-1}]-I_k) ,$$
where $I_k$ is same as above and  we have the inclusions other way around
$$I_0^c\supseteq I_1^c \supseteq \ \ . \ \ . \ \ . \ \ I_k^c\supseteq I_{k+1}^c\supseteq .\ .$$
We will use induction on $n$;\\For $n=1$, it is trivial.\\
Suppose there exist a square free Stanley decomposition of $J^c$ for
a square free monomial ideal $J\subset S'=K[x_1,x_2,...,x_{n-1}]$.\\
Consider $I\subset S=K[x_1,x_2,...,x_{n}]$ be a square free monomial
ideal, by Janet's algorithm$$I^c=\bigoplus_k x_n^k I_k^c \ \ , $$ where
each $I_k^c\subset S'=K[x_1,x_2,...,x_{n-1}]$, it has a square free
Stanley decomposition as
$$I_k^c=\bigoplus_{i_k=1}^{r_k}u_{i_k}K[Z_{i_k}] \ \ \ \ \ \ \ \ \ \ for \ all\ k.$$
Janet algorithm gives the Stanley decomposition for different cases
 as follows:\\
\textbf{ (C1)} When $\alpha \neq \beta$ ($\alpha=0$ and $\beta=1$),
so the Stanley decomposition of $I^c$ will be of the form:
$$I^c=\ (\bigoplus_{i_{\alpha}=1}^{r_{\alpha}} u_{i_{\alpha}}K[Z_{i_{\alpha}}])\ \bigoplus \ (\bigoplus_{i_{\beta}=1}^{r_{\beta}} u_{i_{\beta}}x_n K[Z_{i_{\beta}},x_n])$$
as $supp (u_{i_{\beta}})\in \{Z_{i_{\beta}}\}$ $\Rightarrow$ $supp
(u_{i_{\beta}}x_n)\in \{Z_{i_{\beta}} ,x_{n}\}$ and
$u_{i_{\beta}}x_n$ remain square free as $u_{i_{\beta}}$ is square
free in $S'$. Hence it is a square free Stanley decomposition of
$I^c$.\\
\textbf{(C2)}When $\alpha = \beta = 0$ , the Stanley decomposition
of $I^c$ will be of the form:
$$I^c=\ \bigoplus_{i_{\beta}=1}^{r_{\beta}} u_{i_{\beta}} K[Z_{i_{\beta}},x_n]$$
It is clearly a square free Stanley decomposition of $I^c$.\\
\textbf{(C3)}When $\alpha = \beta =1$ ,  the Stanley decomposition
of $I^c$ will be of the form:
$$I^c=\ K[x_1,x_2,...,x_{n-1}]\bigoplus(\bigoplus_{i_{\beta}=1}^{r_{\beta}} u_{i_{\beta}}x_n K[Z_{i_{\beta}},x_n]) . $$
\end{proof}

This lemma gives a motivation to describe the Janet's algorithm for
the partitions of simplicial complexes.

\section{Janet's algorithm for the partition of simplicial\\
complexes} We will describe the algorithm for the partition of
simplicial complex $\Delta$ on $[n]$ in the view of above lemma.
\begin{Lemma}\label{23}
Janet's algorithm gives a partition of a simplicial complex $\Delta$
on $[n]$ recursively.
\end{Lemma}
\begin{proof}
For any simplicial complex $\Delta $ on $[n]$, we can write
$$\Delta_{0} = \Delta \cap \Delta_{[n-1]}$$
$$n\Delta_{1} = \Delta \cap n\Delta_{[n-1]} $$
where  $\Delta_{[n-1]}= [\emptyset \ ,\ \{123...(n-1) \}]$ and \
$n\Delta_{[n-1]}$ is the interval $\Delta_{[n-1]}$ shifted
with $n$, namely $n\Delta_{[n-1]}=[n \ , \ \{12..(n-1)n\}] $. \\
 It should be noted that
$\Delta_{0}$ and $\Delta_{1}$ are the simplicial complexes on $[n-1]$.
We use induction on $n$.\\
For $n=1$, there is nothing to prove.\\
Suppose the result holds for $n-1$ i.e, every simplicial complex in
$[n-1]$ has a computed partition.\\
Consider $\Delta$ on $[n]$, by the Janet's algorithm
$$\Delta =\Delta_{0}\ \sqcup \ n\Delta_{1}$$
where $\Delta_{0}$ and $\Delta_{1}$ are the simplicial complexes on
$[n-1]$, so there exist their partitions:
$$\Delta_{0} \ = \ \bigsqcup_{i_0=1}^{r_0}[F_{i_0}\ ,\ G_{i_0}]$$
 $$\Delta_{1} \ = \ \bigsqcup_{i_1=1}^{r_1}[F_{i_1}\ ,\ G_{i_1}].$$
 Janet's algorithm gives the partition of $\Delta$ for different
 cases as follows:\\
 \textbf{(C1)} When $\Delta_{0}\ \neq \ \Delta_{1}$ and $\Delta_{0}\ \neq \ \Delta_{[n-1]}$, then the
 partition of $\Delta$ will be of the form
 $$\Delta \ = \ (\bigsqcup_{i_0=1}^{r_0}[F_{i_0}\ ,\ G_{i_0}])\ \bigsqcup \ (\bigsqcup_{i_1=1}^{r_1}[nF_{i_1}\ ,\ nG_{i_1}]).$$  \textbf{(C2)} When $\Delta_{0}\ = \ \Delta_{1}$, then the
 partition of $\Delta$ will be of the form
 $$\Delta \ = \ (\bigsqcup_{i_0=1}^{r_0}[F_{i_0}\ ,\ nG_{i_0}]).$$
\textbf{(C3)} When $\Delta_{0}\ = \ \Delta_{[n-1]}$, then the
 partition of $\Delta$ will be of the form
 $$\Delta \ = \ [\emptyset \ ,\ \{123...(n-1)\}]\ \bigsqcup \ (\bigsqcup_{i_1=1}^{r_1}[nF_{i_1}\ ,\ nG_{i_1}]).$$
 \end{proof}
The following example shows how the  Janet's algorithm works to
compute the partition of a simplicial complex $\Delta$.
\begin{Example}

 \em{Let $\Delta$ be a simplicial complex given by the facets;
 $$\Delta=<\{124\}, \{126\},  \{135\}, \{143\}, \{156\}, \{245\}, \{236\}, \{235\}, \{346\}, \{456\} >$$
  Now by applying the Janet's algorithm,\\
  $$\Delta_0\ =\ \Delta \cap \Delta_{[5]} =\ <\{124\}, \{135\}, \{143\}, \{245\}, \{235\}>$$
$$6\Delta_1=\ \Delta \cap  6\Delta_{[5]} =\ < \{126\}, \{156\}, \{236\}, \{346\}, \{456\} >$$

Now consider $\Delta_0$ in $[5]$, we will use the Janet's algorithm
to find its partition.
$$\Delta_0\ =\ <\{124\}, \{135\}, \{143\}, \{245\}, \{235\}>$$
by applying Janet's algorithm,
$$\Delta'_{00}\ =\ \Delta_0 \cap \Delta_{[4]} =\ <\{124\}, \{143\},  \{23\}>$$
$$5\Delta'_{01}\ =\ \Delta_0 \cap 5\Delta_{[4]} =\ < \{135\},  \{245\}, \{235\}>$$
Partition of $\Delta'_{00}$ will be as follows;
$$\Delta'_{00}\ =\ [\emptyset  ,\{124\}]\sqcup [\{3\} , \{143\}]\sqcup [ \{23\},\{23\}]$$
Partition of $\Delta'_{01}=< \{13\},  \{24\}, \{23\}>$ will be as
follows;
$$\Delta'_{01}\ =\ [\emptyset, \{13\}]\sqcup [\{4\}, \{24\}] \sqcup [\{2\} \{23\}]$$
Hence the partition of $\Delta_0$ by Janet's algorithm is as
follows;
$$\Delta_0\ =\ [\emptyset  ,\{124\}]\sqcup [\{3\} , \{143\}]\sqcup [ \{23\},\{23\}]\sqcup [\{5\}, \{135\}]\sqcup [\{45\}, \{245\}] \sqcup [\{25\} \{235\}]$$
\smallskip
Now consider $\Delta_1$ in $[5]$, we will use the Janet's algorithm
to find its partition.
$$\Delta_1=< \{12\}, \{15\}, \{23\}, \{34\}, \{45\} >$$
by applying Janet's algorithm,
$$ \Delta'_{10}=\ \Delta_1 \cap \Delta_{[4]}\ \ \ and\ \ \ 5\Delta'_{11}=\ \Delta_1 \cap 5\Delta_{[4]}$$
$$\Delta'_{10}=< \{12\}, \{23\}, \{34\} >\ \ \ \ \Rightarrow \ \ \ \Delta'_{10}= [\emptyset , \{12\}] \sqcup [ \{3\} , \{23\}] \sqcup [\{4\} , \{34\}]$$
$$5\Delta'_{11}=<  \{15\},  \{45\} > \ \ \ \ \ \ \ \Rightarrow  \ \ \ \ \ \ \ \ \Delta'_{11}= [\emptyset , \{1\}] \sqcup [ \{4\} , \{4\}] $$
Hence the partition of $\Delta_1$ by Janet's algorithm is as
follows;
$$\Delta_1= [\emptyset , \{12\}] \sqcup [ \{3\} , \{23\}] \sqcup [\{4\} , \{34\}]\sqcup [\{5\} , \{15\}] \sqcup [ \{45\} , \{45\}]$$
consequently, we have the partition of $\Delta$
$$\Delta\ =\ [\emptyset  ,\{124\}] \sqcup [\{3\} , \{143\}]\sqcup [ \{23\},\{23\}]\sqcup [\{5\}, \{135\}]\sqcup
 [\{45\}, \{245\}] \sqcup [\{25\} \{235\}]$$ $$\sqcup \  [\{6\} , \{126\}] \sqcup [ \{36\} , \{236\}] \sqcup [\{46\} , \{346\}]\sqcup [\{56\} , \{156\}] \sqcup [ \{456\} , \{456\}]\ \ \ \Box $$}

\end{Example}

\begin{Remark}
\em{In the above example, it is clear that the partition obtained from Janet's algorithm is not a $nice\ partition$.
Note that $\Delta$ in the above example is in fact the simplicial complex given by the triangulation of the real projective plane and it has a $nice\ partition$ see \cite[Example 22]{Po}. So it is not possible to obtain always a $nice\ partition$ by Janet's algorithm.}
\end{Remark}

\end{document}